\documentclass{siamart251104}
\usepackage{fullpage}
\usepackage{dkmath}
\usepackage{dkmathssq}
\usepackage{amsmath,amssymb,amsfonts}
\usepackage{mathtools}
\usepackage{subcaption}

\newsiamremark{remark}{Remark}
\newsiamthm{thm}{Theorem}
\newsiamthm{newdef}{Definition}

\def\kth{Department of Mathematics, KTH Royal Institute of Technology, Stockholm, Sweden}

\title{Stabilizing the singularity swap quadrature for near-singular line integrals}

\author{
David Krantz%
    \thanks{\kth\,
    ({\tt davkra@kth.se}).}
    \and
Alex H.~Barnett%
    \thanks{Center for Computational Mathematics, Flatiron Institute, New York, United States of America\,
    ~({\tt abarnett@flatironinstitute.org}).}
    \and
Anna-Karin Tornberg%
    \thanks{\kth\,
   ({\tt akto@kth.se}).}
}

\date{\today}

\begin{document}

\maketitle

\begin{abstract}\label{s:abstract}Singularity swap quadrature (SSQ) is an effective method for the evaluation at nearby targets of potentials due to densities on curves in three dimensions. While highly accurate in most settings, it is known to suffer from catastrophic cancellation when the kernel exhibits both near-vanishing numerators and strong singularities, as arises with scalar double layer potentials or tensorial kernels in Stokes flow or linear elasticity. This precision loss turns out to be tied to the interpolation basis, namely monomial (for open curves) or Fourier (for closed curves). We introduce a simple yet powerful remedy: target-specific translated monomial and Fourier bases that explicitly incorporate the near-vanishing behavior of the kernel numerator. We combine this with a stable evaluation of the constant term which now dominates the integral, significantly reducing cancellation. We show that our approach achieves close to machine precision for prototype integrals, and up to ten orders of magnitude lower error than standard SSQ at extremely close evaluation distances, without significant additional computational cost.
\end{abstract}

\begin{keywords}
Nearly singular, close evaluation, numerical quadrature, catastrophic cancellation, boundary integral equation, translated basis
\end{keywords}

\section{Introduction}\label{s:introduction}
We consider the numerical evaluation of potentials in three dimensions of the form
\begin{equation}
    u(\xx) = \int_\Gamma \mathcal{K}(\xx,\yy)\sigma(\yy)\ds(\yy),
    \label{eq:layer_potential}
\end{equation}
where $\D s$ is the arc-length measure on a smooth open or closed curve $\Gamma\subset\mathbb{R}^3$, $\sigma$ is a smooth (possibly vector-valued) function defined on $\Gamma$. The kernel $\mathcal{K}(\xx,\yy)$, derived from a fundamental solution of an elliptic partial differential equation (PDE), is singular as $\xx\rightarrow\yy$, with typical form
\begin{equation}
    \mathcal{K}(\xx,\yy) \sim \frac{k(\xx-\yy)}{|\xx-\yy|^m},\quad m=1,3,5,
    \label{eq:K}
\end{equation}
for some smooth (scalar or tensor) function $k(\xx-\yy)$, or a sum of such forms. An open curve (often resulting from decomposition of a closed curve into panels) is often discretized using Gauss--Legendre quadrature, whereas the most common global discretization of a closed curve is the spectrally-accurate periodic trapezoidal rule. Although we focus on such line potentials in three dimensions (3D), the ideas introduced in this paper are also applicable to layer potentials in 2D.

Integrals such as \eqref{eq:layer_potential} arise in boundary or line integral formulations of elliptic PDEs, where the solution is represented via a potential as above, while the unknown ``density'' function $\sigma$ is determined by imposing the boundary conditions and solving the resulting integral equation. Once $\sigma$ is known, the solution at any point in the computational domain is obtained by evaluating the associated potential.

This setting occurs, for instance, in non-local slender-body theory (SBT) for modeling thin filaments in Stokes flow \cite{KellerRubinow1976,Johnson1980,gotz2000}, with applications including suspensions of straight \cite{TORNBERG2006172,Saintillan2016}, flexible \cite{TORNBERG2004}, or closed-loop \cite{ueda} fibers in viscous fluids.
A key remaining challenge is the development of numerical methods that can robustly and efficiently handle near-contact hydrodynamic interactions between filaments \cite{STEIN2024102379}.
Similar line integrals appear via the Biot--Savart law in magnetostatics and vortex dynamics, where they describe magnetic or velocity fields induced by current-carrying wires or vortex filaments, respectively. Their accurate numerical evaluation is a topic of current research \cite{schilling23,polanco25}.
The frequency-domain electromagnetic case is relevant for the evaluation of solution fields due to line-integral approximations to scattering by thin conducting wires \cite{alyones11,haslamwire}.
Finally, accurate quadratures for integrals such as \eqref{eq:layer_potential} are needed as computational subroutines within more general 3D quadrature schemes, such as when surface integrals are reduced to line integrals using Stokes' theorem \cite{jiang2024} or special parameterizations \cite{krantz2024}.

In all such cases, potential evaluations are needed not only to solve an integral equation for the density function $\sigma$, but also to evaluate the solution in the domain. The accuracy of these evaluations depends strongly on the relative position of the evaluation (target) point $\xx$ and the source curve $\Gamma$. One typically distinguishes between four regimes: the target point is either on or off $\Gamma$, and either far or close to the source points $\yy\in\Gamma$. The most challenging case is \textit{close evaluation}, where $\xx\notin\Gamma$ but lies very close to it. Here, the integrand becomes sharply peaked, rendering standard quadrature rules inaccurate or inefficient.

This has led to the development of so-called special quadrature methods, which remain accurate even as the target approaches the curve. In 2D, where surface and line integrals are the same thing, this problem is largely considered ``solved'', with well-established methods such as Helsing--Ojala \cite{HELSING2008,Ojala2015} quadrature, quadrature by expansion (QBX) \cite{KLOCKNER2013332,barnett2014}, and density interpolation \cite{DIMlap19}. 
For line integrals in 3D, the set of available techniques is more limited. In simulations of viscous flow around slender bodies, a common approach has been to employ regularized kernels with standard quadrature rules \cite{cortez,TORNBERG2004}, but this is non-convergent unless the 
regularization parameter is taken arbitrarily small. Per-target adaptive quadrature on $\Gamma$ is in contrast convergent, but becomes expensive with many close targets.

\begin{figure}[t]
\centering
\begin{minipage}[c]{0.25\textwidth}
\begin{subfigure}[t]{\linewidth}
    \centering
    \includegraphics[width=\linewidth]{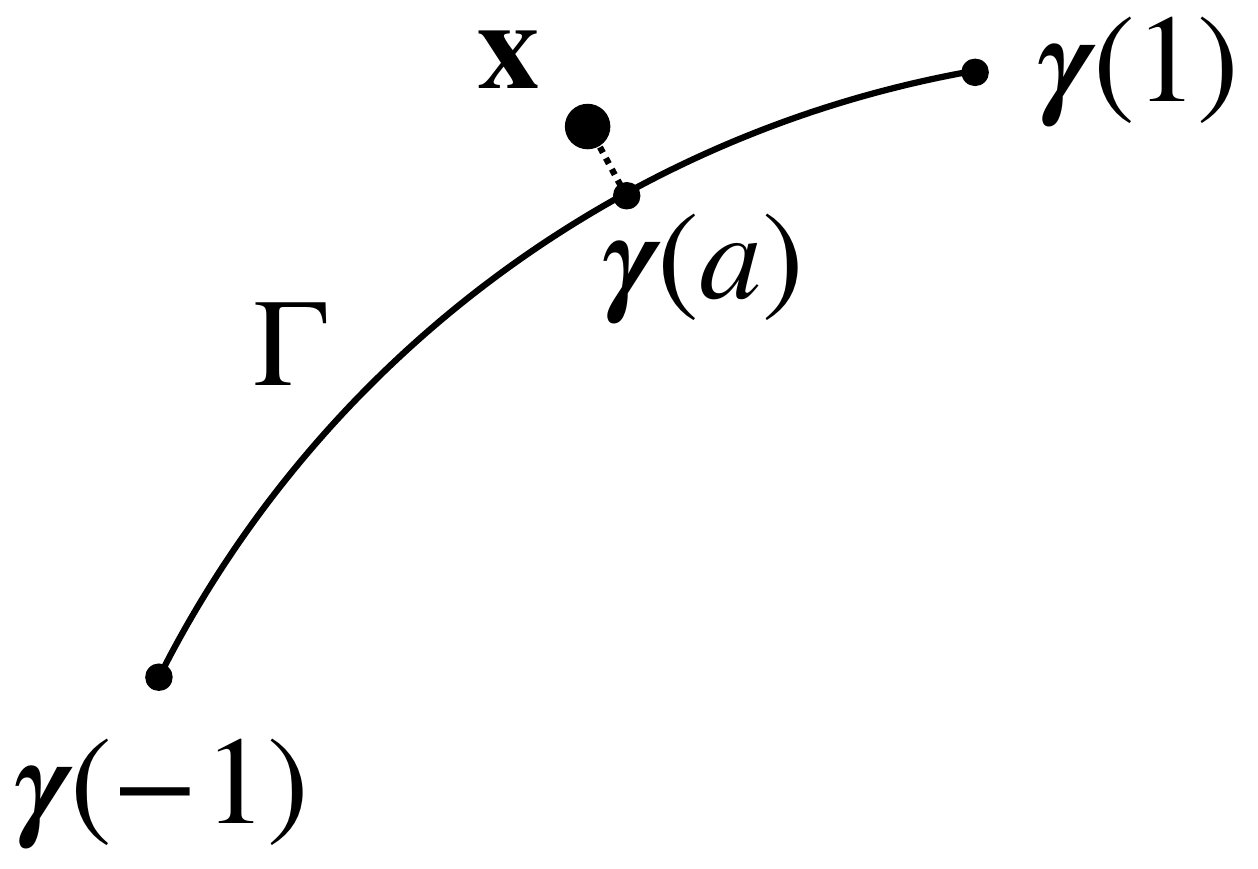}
    \caption{}
\end{subfigure}
\end{minipage}%
\hfill
\begin{minipage}[c]{0.7\textwidth}
    \centering
    \begin{subfigure}[t]{\linewidth}
        \centering
        \includegraphics[width=\linewidth]{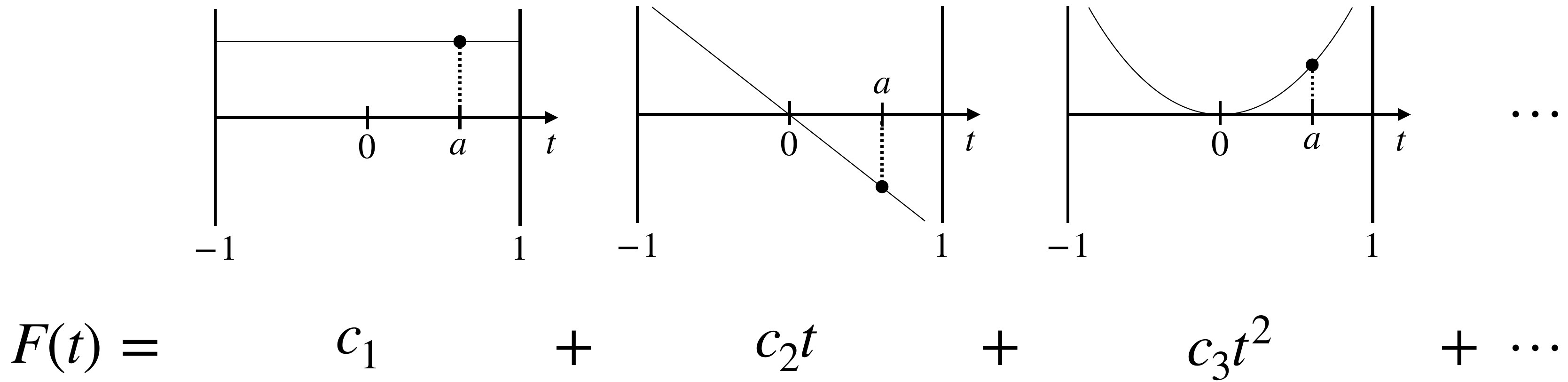}
        \caption{Standard SSQ}
    \end{subfigure}
    \vspace{0.02\linewidth}
    \begin{subfigure}[t]{\linewidth}
        \centering
        \includegraphics[width=\linewidth]{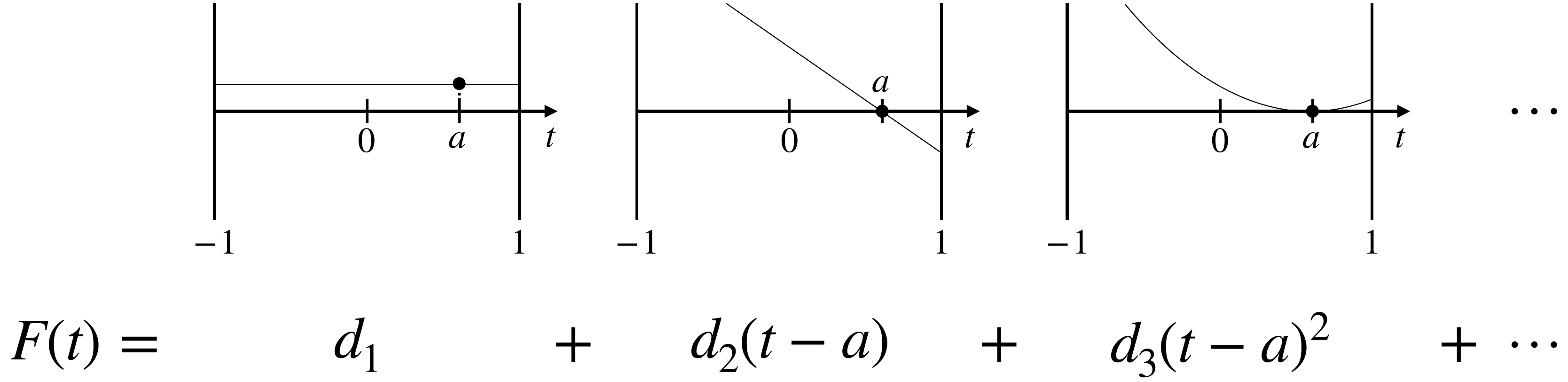}
        \caption{TSSQ}
    \end{subfigure}
\end{minipage}
\caption{Sketch of the principal idea of this paper, in the open-arc case.
(a) shows a line integral along the curve $\Gamma = \ggamma([-1,1]) \subset \RR^3$, with a target point $\xx$ close to the curve point $\ggamma(a)$ ($a$ is the real part of the complex root $t_0$ introduced in \eqref{eq:root_def}).
(b) shows the standard SSQ expansion of the integrand numerator $F(t)$ in the monomial basis: the small value $F(a)$---which dominates the near-singular integral---is given by the sum of larger cancelling terms, making the method unstable.
(c) shows the proposed expansion in a {\em translated} monomial basis,
which we call TSSQ (translated SSQ): now $F(a)$ is captured directly by its small
constant term, computed {\em stably} from the density, while the remaining
terms each contribute much less to the line integral. Catastrophic
cancellation is thereby avoided. \label{fig:idea}}
\end{figure}

A recent alternative is the \textit{singularity swap quadrature} (SSQ), introduced in \cite{AFKLINTEBERG2021}. The method ``swaps'' the near-singular factor for a simpler one having the same singular behavior in the complexified parameter plane. The smooth part of the integrand is interpolated in a fixed basis (typically monomials or Fourier modes), and the resulting \textit{basis integrals}, involving the basis functions multiplied with the simpler nearly singular factor, are evaluated analytically, usually through recurrence relations. This results in a convergent scheme that needs no per-target adaptivity. Originally developed for open curves in two and three dimensions, discretized with panel-based Gauss--Legendre quadrature, SSQ has since been extended to closed curves in two \cite{afKlinteberg2024} and three dimensions \cite{krantz2024}, discretized using the trapezoidal rule. Additionally, \cite{bao2024} introduced a variant of SSQ for closed curves in 2D.

While SSQ has proven highly accurate in many settings, it suffers from \textit{catastrophic cancellation} when the kernel numerator has a near-vanishing factor, as occurs in double layer kernels containing $\nn\cdot(\xx-\yy)$ (where $\nn$ is a normal vector) and tensor-valued kernels (Stokes, elasticity, etc.) containing, for example,  $(\xx-\yy)(\xx-\yy)^T$. In such cases, the true value of the integral is small, while the quadrature sum consists of large, oscillatory terms, leading to significant round-off error. This effect, noted in \cite[Remark 11]{AFKLINTEBERG2021}, \cite[Remark 6.1]{krantz2024}, and \cite[Section 3.1]{bao2024}, becomes more severe with increasing singularity strength and is not resolved by singularity subtraction techniques, as proposed in \cite{bao2024}.

This paper introduces a simple and effective remedy. The core issue is that standard SSQ uses a fixed, generic interpolation basis that does not account for the local vanishing behavior of the numerator of the integrand. We propose instead to adapt the basis to the target point. When the numerator vanishes near a point, so should all but one of the basis functions that interpolate it. For open curves, this leads to monomials expanded about the near-vanishing point, as sketched in Figure~\ref{fig:idea}. For periodic curves, we construct a translated Fourier basis with the same localization.
In both settings, the result is a quadrature whose terms are appropriately scaled, and cancellation is prevented. A key ingredient is the stable recovery of 
the coefficient of the one nonvanishing basis function (the constant function in our settings).

\begin{remark}
Target-adapted monomial bases also play a key role in the corrected trapezoidal rules developed by Nitsche \cite{nitsche_evaluation_2021,nitsche_corrected_2022} for line integrals in 2D, and recently generalized to surface integrals in 3D \cite{nitsche2025correctedtrapezoidalrulesnearsingular}.
They use a double expansion over both shifted monomials and singularity powers, with each term integrated either in closed form or by recurrence relations, and the achieved order of accuracy controlled by the number of retained terms.
Near-singular contributions are then \textit{subtracted}
from the integrand to give fixed-order trapezoid corrections.
In contrast, in our approach the integrand is \textit{divided by} one near-singular function with the correct singularity location (as in SSQ), leaving only a \textit{single} shifted monomial or Fourier expansion to fully approximate the resulting smooth function.
\end{remark}

The outline of this paper is as follows. In Section \ref{s:ssq}, we review the SSQ method and highlight the origin of its catastrophic cancellation. Section \ref{s:modified} introduces our modified interpolation bases, including their new recurrence formulae and properties. Section \ref{s:numerical_experiments} presents numerical experiments demonstrating the improved accuracy when using the modified bases in close evaluation problems. We conclude in Section \ref{s:conclusion},
and a short appendix contains elliptic integrals needed for the Fourier recurrences.

\section{Problem setup and overview of singularity swap quadrature}\label{s:ssq}
We begin by introducing the general class of integrals that we aim to evaluate accurately in the remainder of the paper. We also summarize the key components of the \textit{singularity swap quadrature} (SSQ) method, following the approach introduced in \cite{AFKLINTEBERG2021}. The method is first presented in a general setting, after which we specify the adaptations required for open and closed curves.

Let the analytic curve $\Gamma$ in \eqref{eq:layer_potential} be parametrized by the analytic map $\ggamma:E\rightarrow\mathbb{R}^3$, where the parameter domain $E\subset\mathbb{R}$ is typically chosen as either $[-1,1]$ for open curves, or $[0,2\pi)$ for closed (periodic) curves. The parametrized form of the line potential \eqref{eq:layer_potential} with a general power-law singularity \eqref{eq:K} is
\begin{equation}
I_m = I_m(\xx) = \int_E \frac{\ftilde(t,\xx)\sigmatilde(t)}{|\ggamma(t)-\xx|^m}\dt,\quad m=1,3,5,
\label{eq:Imparam}
\end{equation}
where $\sigmatilde(t)\coloneqq\sigma(\ggamma(t))$ is the pullback of the density function, while $\ftilde(t,\xx)$ includes parametrization and smooth kernel factors
that can be evaluated at any $t\in E$. This is accurately approximated using a standard quadrature rule with nodes $\{t_j\}_{j=1}^n$ when the target point $\xx$ is sufficiently far from the curve.

To develop a robust quadrature for close evaluation, we introduce the squared-distance function between the target point $\xx=(x_1,x_2,x_3)$ and the curve $\Gamma = \{\ggamma(t)~\vert~t\in E\}$, where $\ggamma(t) = (\gamma_1(t),\gamma_2(t),\gamma_3(t))$,
\begin{equation}
    R(t)^2 \coloneqq \left|\ggamma(t)-\xx\right|^2 = \left(\gamma_1(t)-x_1\right)^2+\left(\gamma_2(t)-x_2\right)^2+\left(\gamma_3(t)-x_3\right)^2.
    \label{eq:R2}
\end{equation}
Using this, we rewrite the integral as
\begin{equation}
    I_m = \int_E \frac{\ftilde(t,\xx)\sigmatilde(t)}{\left(R(t)^2\right)^{m/2}}\dt,
    \label{eq:Im}
\end{equation}
The key idea in \cite{AFKLINTEBERG2021} is to transform the integral \eqref{eq:Im} over a curve in $\mathbb{R}^3$ into an integral with a known singularity structure on a simple domain in $\mathbb{C}$. This is accomplished by analytically continuing $R(t)^2$ from $E$ to a (usually large) open neighborhood in the complex $t$ plane\footnote{When $R(t)^2$ is evaluated with complex arguments, it is computed according to the right-most expression in \eqref{eq:R2}.} and identifying a pair of complex conjugate roots $t_0(\xx)$, $\overline{t_0(\xx)}$ such that
\begin{equation}
    R(t_0(\xx))^2=R(\overline{t_0(\xx)})^2=0.
    \label{eq:root_def}
\end{equation}
We refer the reader to \cite[Figure~5]{AFKLINTEBERG2021} for a visualization of these root pairs and the corresponding target point $\xx$.
We assume that these roots lie in the domain of analyticity of $\ggamma$ and hence of $R^2$ (and this explains why we write expressions using $R^2$ and not $R$, following \cite{AFKLINTEBERG2021}).
From now we write simply $t_0$, leaving the target-dependence implied.
In practice, given $\ggamma(t)$ and $\xx$, $t_0$ is easily found through a numerical complex-plane root-finding procedure, to which we refer the reader to \cite[Section~3.2]{AFKLINTEBERG2021} and \cite[Section~2.1]{afKlinteberg2024} for details. 

This yields the reformulation
\begin{equation}
    I_m = \int_E \frac{F(t)}{h(t,t_0)^m}\dt,
    \label{eq:Imswap}
\end{equation}
where the ``SSQ numerator'' $F(t)$ is defined by
\begin{equation}
    F(t) \coloneqq \ftilde(t,\xx)\sigmatilde(t)\frac{h(t,t_0)^m}{\left(R(t)^2\right)^{m/2}},
    \label{eq:F}
\end{equation}
and the function $h(t,t_0)$ is chosen such that $h(t,t_0)^2$ matches the root structure of $R(t)^2$, so that $F(t)$ becomes smooth on $E$ (in fact analytic in the large open neighborhood of $E$).\footnote{Alternative swapping strategies, such as the subtraction method used in \cite{bao2024}, also exist.}

To evaluate $I_m$, one interpolates the SSQ numerator $F(t)$ using some fixed basis $\{\varphi_k\}_{k=1}^n$, so that
\begin{equation}
F(t) \approx \sum_{k=1}^n c_k(t_0)\varphi_k(t),\quad t\in E,
\label{eq:Fapprox}
\end{equation}
with coefficients $c_k(t_0)$ computed from numerator samples $F_j\coloneqq F(t_j)$.
This yields the interpolatory quadrature formula
\begin{equation}
    I_m \approx \sum_{k=1}^n c_k(t_0)\underbrace{\int_E\frac{\varphi_k(t)}{h(t,t_0)^m}\dt}_{\eqqcolon B_k^m(t_0)}
    \; =: \; \mathbf{c}^T\mathbf{B}^m,
    \label{eq:Imapprox}
\end{equation}
where $\mathbf{c}$ and $\mathbf{B}^m$ denote the column vectors with entries $\{c_k\}$ and $\{B_k^m\}$, respectively. We refer to $B_k^m$ as \textit{basis integrals}---these are typically evaluated analytically using recurrence relations---while the associated vector $\mathbf{c}\odot\mathbf{B}^m$, where $\odot$ denotes the Hadamard (elementwise) product, is the quadrature vector.

\begin{table}[t]
\centering
\caption{SSQ nodes, interpolation bases, and singularity-cancelling functions whose squares match the root structure of $R(t)^2$ in \eqref{eq:R2}, for the two types of curves considered.}
\label{tab:ssq_summary}
\begin{tabular}{|c||l|l|l|}
\hline
Curve, $\Gamma$                                            & \multicolumn{1}{c|}{Discretization nodes, $\{t_j\}_{j=1}^{n}\in E$} & \multicolumn{1}{c|}{Basis functions, $\varphi_k$} & \multicolumn{1}{c|}{\begin{tabular}[c]{@{}c@{}}Function\\$h(t,t_0)$\end{tabular}} \\ \hline\hline
\begin{tabular}[c]{@{}c@{}}Open\\(non-periodic) \cite{AFKLINTEBERG2021}\end{tabular}                                                      & Gauss--Legendre, $E=[-1,1]$                           & Real monomial, $t^{k-1},~k=1,\dots,n$                          & $|t-t_0|$                                                                                           \\ \hline
\begin{tabular}[c]{@{}c@{}}Closed\\ (periodic) \cite{krantz2024}\end{tabular} & Periodic trapezoidal, $E=[0,2\pi)$                             & Fourier, $e^{ikt},~
-n/2\le k<n/2$
& $|e^{it}-e^{it_0}|$                                                                                 \\ \hline
\end{tabular}
\end{table}

To make this framework concrete, we must specify the discretization nodes $\{t_j\}$, basis functions $\{\varphi_k\}$, and the form of the function $h(t,t_0)$. Following \cite{AFKLINTEBERG2021,afKlinteberg2024,krantz2024}, we adopt the choices summarized in Table~\ref{tab:ssq_summary}.
For the real monomial basis, the coefficients $c_k$ in \eqref{eq:Fapprox} are obtained by solving a Vandermonde system via the Björck--Pereyra algorithm \cite{bjorck} or using a precomputed LU factorization, both with computational cost $\mathcal{O}(n^2)$. For the Fourier basis, the coefficients, corresponding to the Fourier coefficients of $F(t)$ in \eqref{eq:F}, are computed using the FFT at an $\mathcal{O}(n\log n)$ cost.
Examining the table, the basis integrals $B_k^m(t_0)$ are as follows:
\begin{itemize}
    \item For open curves (real monomial basis),
    \begin{equation}
    P_k^m(t_0) = \int_{-1}^{1} \frac{t^{k-1}}{|t-t_0|^m}\dt,\quad k=1,2,\dots,n;
    \label{eq:monomial_integral}
    \end{equation}
    \item For closed curves (Fourier basis),
    \begin{equation}
    S_k^m(t_0) = \int_{0}^{2\pi} \frac{e^{ikt}}{\left|e^{it}-e^{it_0}\right|^m}\dt,\quad k\in\mathbb{Z},
    \label{eq:fourier_integral}
    \end{equation}
\end{itemize}
both of which admit accurate and efficient evaluation via recurrence formulas, for each power $m=1,3,5$. For the monomial integrals \eqref{eq:monomial_integral}, recurrence formulas are given in \cite[Appendix B]{TORNBERG2006172}, with improved stabilization for the initial values $k=1$ presented in \cite[Section 3.1]{AFKLINTEBERG2021}. For the Fourier integrals \eqref{eq:fourier_integral}, recurrences are provided in \cite[Lemma 2.3]{krantz2024}, along with a stabilization technique described in \cite[Appendix B]{krantz2024}.
We note that the basis integrals in \eqref{eq:monomial_integral}-\eqref{eq:fourier_integral}, as well as the modified basis integrals introduced in Section \ref{s:modified}, admit closed-form expressions. However, we choose to evaluate them using recurrence relations, as this approach is computationally more efficient.

In exact arithmetic, the combination of the interpolatory quadrature \eqref{eq:Imapprox} and recurrence relations for the basis integrals yields a formula for $I_m$ \eqref{eq:Imparam} whose accuracy is governed by the interpolation error of $F(t)$ in \eqref{eq:Fapprox}. 
In floating-point arithmetic, however, \textit{catastrophic cancellation} may occur in \eqref{eq:Imapprox} when $F(t)$ is near-vanishing for $\ggamma(t)$ close to the target point $\xx$.
This situation typically arises when $k(\xx-\yy)$ contains one or more factors $(\xx-\yy)$, so that $F(t)$ has a root near the real parameter $t = a\coloneqq\Re(t_0)$. The observed cancellation has two coupled causes.

First, if $F$ has one or more roots near $t=a$, then representing $F(t)$ in the basis $\{\varphi_k\}$, whose elements are generally of size $\mathcal{O}(1)$ near $t=a$, forces the small value of $F(t)$ for $t\approx a$ to be produced by cancellation among $\mathcal{O}(1)$ contributions.
To illustrate, suppose that $F(t)\sim(t-a)^3$ to leading order.
Evaluating this factor via its expanded polynomial form in the left-hand side of the following equation,
\begin{equation}
t^3-3t^2a+3ta^2-a^3 = (t-a)^3,
\label{eq:poly}
\end{equation}
incurs severe loss of significance for $t\approx a$, whereas evaluating it in its factored form in the right-hand side does not \cite[Section 1.4.3]{van_loan}.

Second, the denominator in \eqref{eq:Imapprox} concentrates the contribution to the integral $I_m$ around $t\approx a$.
Consequently, the terms $c_k P^m_k$ or $c_k S^m_k$ are typically much larger in magnitude than $I_m$, so that the final value is obtained only through strong cancellation in the quadrature sum.
Equivalently: although the interpolant of $F(t)$ may have small \textit{absolute} error throughout the domain $E$,
its \textit{relative} error around $t\approx a$ is arbitrarily large because
$F$ is near-vanishing there, and since this region dominates the integral, it controls the overall error in $I_m$.\footnote{A simple analogy is trying to numerically integrate the smooth function $(\sin x)/(\pi-x)$ by replacing $\sin x$ by its Taylor series about $x=0$ then integrating term by term. Each term involves blow-up at $x=\pi$ even though the true function does not.}

In the next section, we introduce alternative interpolation bases, centered at $a$, that significantly improve the numerical stability of SSQ and reduce cancellation error to a negligible level.
In practice, we decide between the standard and modified bases using a simple criterion based on the location of $t_0$ relative to the integration interval $E$, as described in Section \ref{ss:when_to_use_mod_basis}.
Nevertheless, the approximate cancellation error introduced in the following remark remains a useful tool for analyzing and illustrating the underlying cancellation behavior.

\begin{remark}[Numerical cancellation error]\label{rem:cancellation_errest}Loss of accuracy in a nearly cancelling sum can be understood via the condition number of the summation operation. Writing the sum as the dot product $S=\mathbf{a}^T\mathbf{b}$ with the real-valued column vectors $\mathbf{a}$ and $\mathbf{b}$, the condition number of $S$ with respect to perturbations in $\mathbf{a}$ or $\mathbf{b}$ is
\begin{equation}
    \kappa_{\textrm{dot}}(\mathbf{a},\mathbf{b}) = \frac{\|\mathbf{a}\|\|\mathbf{b}\|}{|\mathbf{a}^T\mathbf{b}|},
\end{equation}
where, for example, we may choose the maximum norm. In finite precision arithmetic, the relative error in $S$ due to sensitivity with respect to the input vectors can then be approximated as
\begin{equation}
E_{\textrm{cancel}}=\kappa_{\textrm{dot}}(\mathbf{a},\mathbf{b})\epsilon_{\textrm{mach}},
\end{equation}
where $\epsilon_{\textrm{mach}}$ denotes the double precision machine rounding error. This means that large errors arise when $\mathbf{a}$ and $\mathbf{b}$ are nearly orthogonal.
\end{remark}

\section{Modified interpolation bases}\label{s:modified}
We now introduce translated interpolation bases for the function $F(t)$ in \eqref{eq:F}, designed to reduce numerical cancellation in the quadrature evaluation \eqref{eq:Imapprox} when the target point is close to the curve $\Gamma$, and the kernel $\mathcal{K}$ in \eqref{eq:layer_potential} involves near-vanishing numerator factors.
The \textit{modified basis} should satisfy the following design principles:
\begin{enumerate}
    \item The quadrature sum is dominated by at most one mode in the basis expansion.
    \item All non-constant basis functions should vanish at $a$, the real argument where $F(t)$ is near-vanishing.
    \item The basis spans the same space as the standard one and retains a comparable approximation power.
\end{enumerate}
Elaborating on the 2nd design principle, the rate (power) of vanishing should
ideally match that of $F(t)$. However,
as we will see in Remark \ref{rem:modified_fourier_coeff}, an exact match is not always necessary for the method to be effective in practice.

We will also later highlight the importance of computing the coefficient associated with the constant basis function 
(the ``one mode'' in the 1st design principle) 
with high accuracy. As we will see, inaccuracies in this coefficient can severely limit the overall precision, even when using a well-designed basis.
In addition, we introduce a stable construction of target-specific quadrature weights, designed specifically for the modified bases.

We now turn to the construction of such modified bases, beginning with the open curve case, followed by the periodic setting. Recall that $F(t)$ is defined after ``swapping'' the singularity in the integrand \eqref{eq:Imparam}, and depends implicitly on the complex-valued root $t_0=a+ib$ (with $a,b\in\mathbb{R}$, $t_0\notin[-1,1]$, or $e^{it_0}$ not on the unit circle) of the squared distance function $R(t)^2$ in \eqref{eq:R2}.

\subsection{Open curves: translated monomial basis}\label{ss:open_curves}
For open curves, the modification may be a simple translation of the origin.

\begin{newdef}[Translated monomial basis]\label{def:modified_monomial_basis}Let $a\in\mathbb{R}$ be a shift parameter. For a given integer $n>0$, the translated (or modified) monomial basis of order $n$ centered at $a$, is defined as
\begin{equation}
    \left\{(t-a)^{k-1}\right\}_{k=1}^{n}.
\end{equation}
\end{newdef}

Expanding $F(t)$ in this basis, we obtain the approximation
\begin{equation}
    F(t) \approx \sum_{k=1}^n d_k(t_0)(t-a)^{k-1},\quad t\in[-1,1],
    \label{eq:Fexp_shifted_monomial}
\end{equation}
where the coefficients $d_k(t_0)$, $k>1$, are determined by solving a Vandermonde system based on the samples $\{F_j\}$, with computational cost $\mathcal{O}(n^2)$. This leads to the modified quadrature approximation
\begin{equation}
    I_m \approx I_m^{\textrm{mod}} = \sum_{k=1}^nd_k(t_0)\widetilde{P}_k^m(t_0) = \mathbf{d}^T\mathbf{\widetilde{P}}^m,
    \label{eq:Imapprox_shifted_monomial}
\end{equation}
where we now define the basis integrals $\widetilde{P}_k^m(t_0)$ and show their
computation via modified recurrence relations.

\begin{lemma}[Modified monomial basis integrals]\label{lem:mod_monomial_rec}Let $k\in\mathbb{N}^{+}$, $m\in\{2n+1:n\in\mathbb{N}\}$, $t_0=a+ib$, with $a,b\in\mathbb{R}$ and $t_0\notin[-1,1]$. Then the modified basis integrals
\begin{equation}
    \widetilde{P}_k^m(t_0) := \int_{-1}^{1} \frac{(t-a)^{k-1}}{|t-t_0|^m}\dt
    \label{eq:Pkm}
\end{equation}
can be computed via the recurrences
\begin{equation}
    \widetilde{P}_{k+1}^m(t_0) =
    \begin{dcases}
        \dfrac{t_2^{k-1}u_2-t_1^{k-1}u_1-(k-1)b^2\widetilde{P}_{k-1}^m(t_0)}{k},& m=1\text{ and } k>1,\\
        \widetilde{P}_{k-1}^{m-2}(t_0)-b^2\widetilde{P}_{k-1}^{m}(t_0), & m>1\text{ and } k>1,
    \end{dcases}
    \label{eq:pkm_rec}
\end{equation}
where $t_1=-1-a$, $t_2=1-a$, and
\begin{equation}
    u_1 = \sqrt{(1+a)^2+b^2} = |1+t_0|,\quad u_2=\sqrt{(1-a)^2+b^2}=|1-t_0|.
\end{equation}
Initial values for $m=1,3,5$ are given by
\begin{align}
&\widetilde{P}_1^1(t_0) = \asinh\left(\frac{t_2}{|b|}\right)- \asinh\left(\frac{t_1}{|b|}\right), &&\widetilde{P}_2^1(t_0) = u_2-u_1, \label{eq:pk1_initial} \\
&\widetilde{P}_1^3(t_0) = \frac{1}{b^2}\left(\frac{t_2}{u_2}-\frac{t_1}{u_1}\right), &&\widetilde{P}_2^3(t_0) = \frac{1}{u_1}-\frac{1}{u_2}, \label{eq:pk3_initial} \\
&\widetilde{P}_1^5(t_0) = \frac{1}{3b^2}\left(\frac{t_2}{u_2^3}-\frac{t_1}{u_1^3}+2\widetilde{P}_1^3(t_0)\right), &&\widetilde{P}_2^5(t_0) = \frac{1}{3}\left(\frac{1}{u_1^3}-\frac{1}{u_2^3}\right). \label{eq:pk5_initial}
\end{align}
\end{lemma}

\begin{proof}
Consider the case $m=1$. Let $s=t-a$, so that $\widetilde{P}_k^1(t_0) = \int_{-1-a}^{1-a}\frac{s^{k-1}}{\sqrt{s^2+b^2}}\ds$. Integrate both sides of the identity $\frac{\D}{\D s}(s^k/\sqrt{s^2+b^2}) = ks^{k-1}/\sqrt{s^2+b^2} - s^{k+2}/(s^2+b^2)^{3/2}$ over $s\in[-1-a,1-a]$, which yields
\begin{equation}
\widetilde{P}_{k+2}^3(t_0) = k\widetilde{P}_k^1(t_0) - \left[\frac{s^k}{\sqrt{s^2+b^2}}\right]_{-1-a}^{1-a}.
\label{eq:proof1}
\end{equation}
Also note that
\begin{equation}
\widetilde{P}_{k+2}^1(t_0)=\widetilde{P}_{k+4}^3(t_0)+b^2\widetilde{P}_{k+2}^3(t_0).
\label{eq:proof2}
\end{equation}
Insert \eqref{eq:proof1} into the two terms in the right-hand side of \eqref{eq:proof2}, solve for $\widetilde{P}_{k+2}^1(t_0)$ and shift the index $k$ to get the recurrence formula in \eqref{eq:pkm_rec}.

For $m=3$, it holds that $(t-a)^2/|t-t_0|^3 = 1/|t-t_0| - b^2/|t-t_0|^3$. Multiply both sides with $(t-a)^{k-1}$ and integrate both sides over $[-1,1]$, which gives the recurrence formula in \eqref{eq:pkm_rec}. 
The formulas for odd $m>3$ are derived analogously, and the corresponding initial values can be obtained in closed form from the defining integrals.
\end{proof}

\begin{remark}
The expressions for $\widetilde{P}_1^3(t_0)$ and $\widetilde{P}_1^5(t_0)$ in \eqref{eq:pk3_initial} and \eqref{eq:pk5_initial} may become inaccurate when $t_0$ lies within the cones extending outwards from the endpoints $\pm1$. In such cases, the stabilized expressions in \cite[Section 3.1]{AFKLINTEBERG2021} should be used instead. In that work, the authors also stabilized their expression for $\widetilde{P}_1^1(t_0)$, written as $\log(t_2+u_2)-\log(t_1+u_1)$, to maintain accuracy near $[-1,1]$. By contrast, the expression for $\widetilde{P}_1^1(t_0)$ in \eqref{eq:pk1_initial} remains accurate without additional stabilization.
\end{remark}

The main benefit of this translation is that the magnitudes of the integrals $\widetilde{P}_k^m(t_0)$ for $k>1$ are significantly reduced compared to those of the standard monomial basis (see \eqref{eq:monomial_integral}). 
This can be understood from the fact that the dominant contribution to both the standard and modified basis integrals arises from a neighborhood around $t=a$.
In the standard basis, the monomials $t^{k-1}$ are generally $\mathcal{O}(1)$ in this region (unless $|a|$ is small), whereas the translated basis functions $(t-a)^{k-1}$ vanish at $t=a$ and are therefore small in the dominant region. 
This makes the modified basis integrals more consistent with the actual magnitude of the integral $I_m$, thereby reducing cancellation.
Equivalently, the translation enforces a factored representation of the near-vanishing behavior: it works directly with powers of $(t-a)$, as in the right-hand side of \eqref{eq:poly}, rather than reconstructing them through the expanded polynomial form on the left-hand side.
The case $k=1$, where the modified basis integral reduces to that of the standard basis, needs special treatment, as we now explain.

\subsubsection{Stable evaluation of the constant coefficient \boldmath{$d_1$}}\label{sss:stable_eval}
Thanks to the design of the translated basis, the $k=1$ (constant) term now dominates the overall integral $I_m$.
Yet the large magnitude of $\widetilde{P}_1^m(t_0)$ (relative to $|I_m|$) remains a source of numerical sensitivity, as any error in computing the coefficient $d_1(t_0)$ is much amplified, degrading the overall accuracy.
Since $d_1(t_0)=F(a)$, this coefficient is small due to $\ftilde(t,\xx)$ nearly vanishing at $t=a$. Thus its value obtained from solving the Vandermonde system given the data $\{F_j\}$ cannot in general have high \textit{relative} accuracy.

To address this, we instead \textit{interpolate the density} $\sigmatilde$, then
compute $d_1(t_0)$ directly from known analytic factors.
Letting $\sigmatilde_j=\sigmatilde(t_j)$, with $t_j$ being the Gauss--Legendre nodes on $[-1,1]$, we interpolate $\sigmatilde(a)\eqqcolon\sigmatilde_a$ using barycentric Lagrange interpolation \cite{barycentric}. Then, by recalling that $h(t,t_0)=|t-t_0|$ from Table \ref{tab:ssq_summary}, and \eqref{eq:F}, we set
\begin{equation}
d_1(t_0) = \ftilde(a,\xx)\sigmatilde_a\frac{|a-t_0|^m}{(R(a)^2)^{m/2}},
\label{eq:d1}
\end{equation}
which retains high relative accuracy.

\subsubsection{Tests of effectiveness of the modified monomial basis}\label{sss:test_mod_monomial}
To illustrate the improvement, we consider a simplified integral on $[-1,1]$, discretized by a $20$-point Gauss--Legendre quadrature rule, mimicking the near-vanishing numerator of relevant line potentials. Let $t_0=a+ib$, $a,b\in\mathbb{R}$ and define
\begin{equation}
I_m(t_0) = \int_{-1}^1 \frac{\left((t-a)^2+\delta\right)\sigma(t)}{|t-t_0|^m}\dt=
\int_{-1}^1 \frac{\left((t-a)^2+\delta\right)\sigma(t)}{\left((t-a)^2+b^2\right)^{m/2}}\dt,\quad m=1,3,5,
\label{eq:ex1_int}
\end{equation}
choosing an analytic density $\sigma(t)=\sin(t+1.53)$.
The numerator vanishes quadratically at $t=a$ up to a small offset $\delta>0$.
Specifically, note that the $m=3$ case of \eqref{eq:ex1_int} is an integrand form occurring when using Stokes slender body theory (SBT) (see the $I_3$ component in Section~\ref{ss:Stokes}) on the
straight fiber segment $\ggamma(t)=(t,0,0)$, $t\in[-1,1]$, with nearby evaluation point $(a,b,0)$. In that setting, $\delta$ represents either the square of the target distance, or of the fiber radius (in the doublet correction to SBT).

In the first test, we let $t_0=0.23+ib$ and vary the distance parameter $b$ logarithmically from $10^{-5}$ to $1$, with fixed $\delta=10^{-8}$. The result is presented in Figure \ref{fig:ex1}. Figure \ref{fig:ex1_err_vs_dist} shows that the modified basis strategy maintains errors near machine precision for all $m$, while the standard basis suffers from cancellation errors when $m=3,5$. Figure \ref{fig:ex1_norm_vs_dist} compares the maximum norm of the quadrature vector to $|I_m|$, confirming that the discrepancy in magnitudes correlates with increased error.

\begin{figure}[t!]
\centering
\begin{subfigure}[t]{0.5\textwidth}
\includegraphics[width=\linewidth]{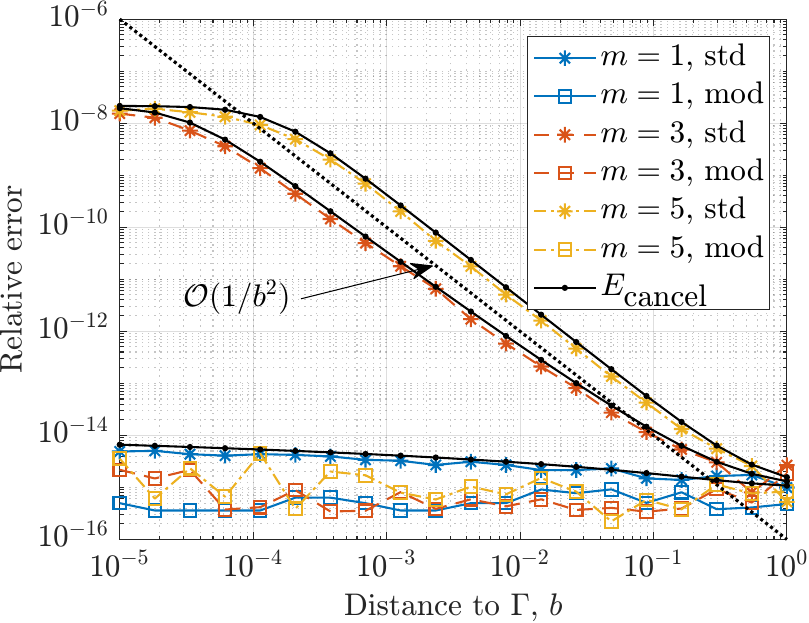}
\caption{}
\label{fig:ex1_err_vs_dist}
\end{subfigure}
\begin{subfigure}[t]{0.49\textwidth}
\includegraphics[width=\linewidth]{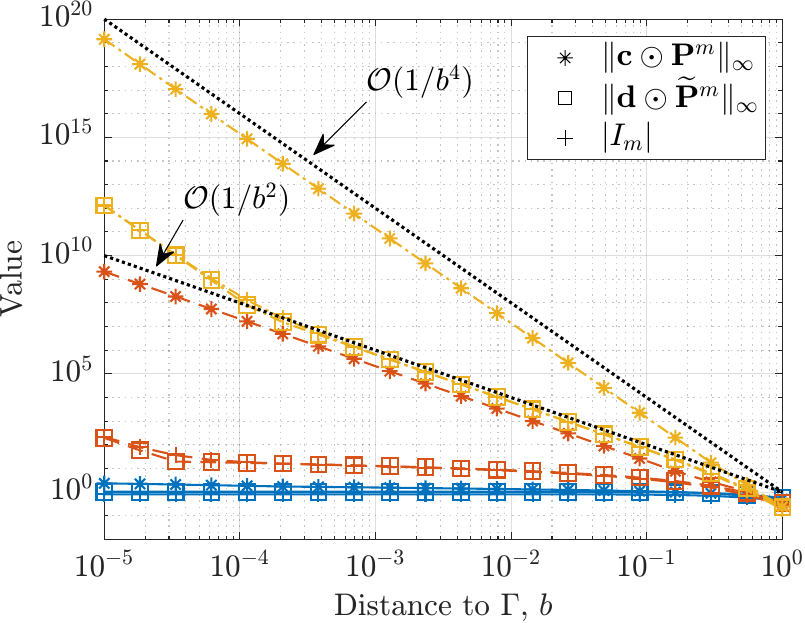}
\caption{}
\label{fig:ex1_norm_vs_dist}
\end{subfigure}
\caption{Panel (a) shows the relative error in computing the integral \eqref{eq:ex1_int} using standard (std) and modified (mod) monomial bases, as the imaginary part of the target point $t_0=a+ib$ decreases. For stronger singularities ($m=3,5$), the standard basis suffers from cancellation, while the modified basis maintains high accuracy. Solid black lines with dots indicate estimated cancellation errors from Remark \ref{rem:cancellation_errest}. Panel (b) shows the maximum norm of the quadrature vector compared to the magnitude of the true integral $I_m$.
For the standard basis with $m>1$, the quadrature vector is orders of magnitude larger than $|I_m|$, providing direct evidence that the loss of accuracy in panel (a) is caused by cancellation. Here, the colors and line types indicate the value of $m$ as in (a).}
\label{fig:ex1}
\end{figure}

In Figure \ref{fig:ex3_err_vs_delta}, we vary $\delta$ and observe that the modified basis yields small errors regardless of how small the numerator becomes, whereas the standard basis performs well only for larger $\delta$. 
Finally, Figure \ref{fig:ex2_err_vs_dist_corr_coeff} demonstrates that stable computation of the coefficient $d_1(t_0)$ is essential for the modified basis to perform well.
In this experiment, the relative error in $d_1(t_0)$ is approximately $1.7 \times 10^{-16}$ when computed using the stabilized formula \eqref{eq:d1}, compared to $5.4 \times 10^{6}$ when obtained directly from the Vandermonde solve.

\begin{figure}[t!]
\centering
\begin{minipage}[t]{.485\textwidth}
\centering
\includegraphics[width=\linewidth]{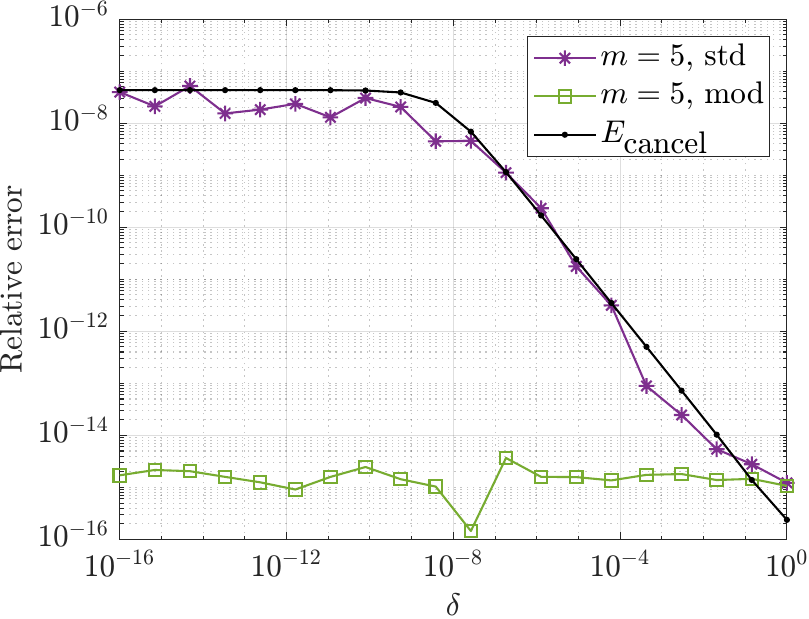}
\caption{Relative error in evaluating the integral \eqref{eq:ex1_int} for fixed $t_0=0.23+10^{-4}i$ and varying $\delta$. The modified (mod) basis maintains high accuracy across all $\delta$, while the standard (std) basis suffers from cancellation when the numerator becomes too small, a behavior that is captured well by the approximate cancellation error from Remark \ref{rem:cancellation_errest}, shown as the solid black line with dots.}
\label{fig:ex3_err_vs_delta}
\end{minipage}%
\hfill
\begin{minipage}[t]{.485\textwidth}
\centering
\includegraphics[width=\linewidth]{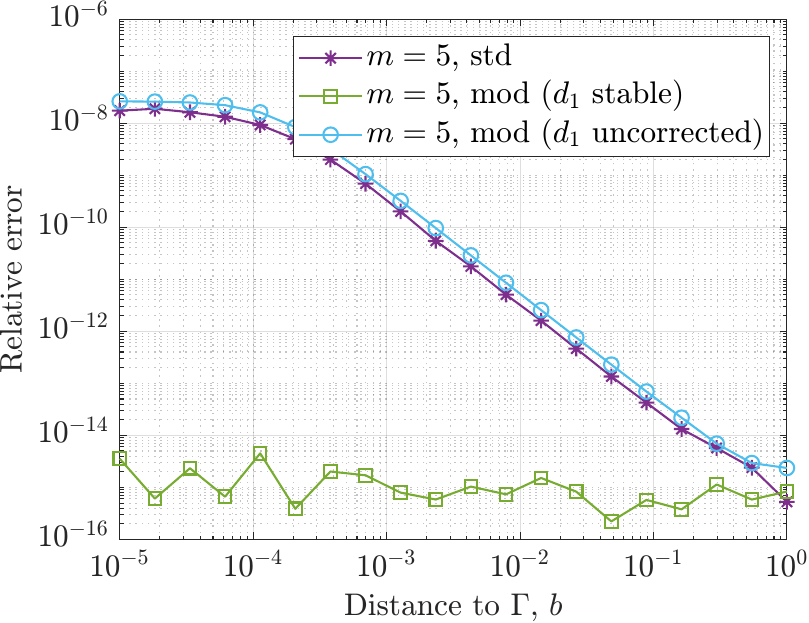}
\caption{Comparison of relative error in evaluating the integral \eqref{eq:ex1_int} using the modified (mod) basis with and without the stable evaluation of the coefficient $d_1(t_0)$. When using the uncorrected value from the Vandermonde solve, the accuracy degrades significantly, resembling that of the standard (std) basis. Results are for fixed $\delta=10^{-8}$.}
\label{fig:ex2_err_vs_dist_corr_coeff}
\end{minipage}
\end{figure}

\subsection{Closed curves: modified Fourier basis}\label{ss:closed_curves}
For periodic problems, unlike for monomials, a simple translation does not result in bases satisfying the three design principles, because the functions $e^{ikt}$ never vanish. We now show how to construct a modified Fourier basis,
where all but one of the basis functions
vanish at $t=a$, yet which span the same space as the standard Fourier basis, and for which there remains a fast transform to evaluate coefficients. A fast transform is crucial because the number of nodes is typically much larger than for an open curve. 

\begin{newdef}[Modified Fourier basis]\label{def:modified_fourier_basis}Let $a\in\mathbb{R}$ be a shift parameter. For an even integer $n>0$, the modified Fourier basis of order $n$ centered at $a$, is defined as
\begin{equation}
\left\{1,~\sin(t-a),~\sin^2\left(\dfrac{t-a}{2}\right)e^{ikt}:k=-n/2+1,\dots,n/2-2\right\}.
\label{eq:mod_fourier_basis_def}
\end{equation}
\end{newdef}

\begin{remark}\label{rem:basis_comparison}The basis in Definition \ref{def:modified_fourier_basis} can be viewed as the periodic analogue of a modified monomial basis $\{1,~t-a,~(t-a)^2t^k\}$, which we have checked performs comparably to the translated monomial basis of Definition \ref{def:modified_monomial_basis} on the simplified example in Section \ref{sss:test_mod_monomial}. However, the latter is more versatile, as it naturally accommodates a broader range of vanishing behaviors.
In principle, one could construct a modified Fourier basis with higher-order vanishing at $t=a$, e.g.~by replacing the oscillatory functions with $\sin^3((t-a)/2)e^{ikt}$ and supplementing with missing modes to ensure completeness. In practice, however, the basis in Definition \ref{def:modified_fourier_basis} has proven sufficient to accurately capture periodic functions with vanishing beyond quadratic order, rendering such complications unnecessary.
\end{remark}

As before, let $t_0=a+ib$ be a root of $R(t)^2$ defined in \eqref{eq:R2}. The $2\pi$-periodic function $F(t)$ in \eqref{eq:F}, sampled at $n$ equispaced nodes with even $n$, is then approximated as
\begin{equation}
F(t) \approx a_0(t_0) + a_1(t_0)\sin(t-a) + \sum\limits_{k=-n/2+1}^{n/2-2} b_k(t_0)\sin^2\left(\dfrac{t-a}{2}\right)e^{ikt},\quad t\in[0,2\pi).
\label{eq:Fexp_mod_fourier}
\end{equation}
The coefficients $\{a_0,a_1,b_k\}$ may be obtained via a sparse linear transformation of the standard Fourier coefficients of $F(t)$, themselves computed efficiently using the FFT, with total cost $\mathcal{O}(n\log n)$.
We give the formulae below in Theorem \ref{thm:mod_coeffs}. 
The expansion in the modified Fourier basis yields the approximation
\begin{equation}
I_m \approx I_m^{\textrm{mod}} = a_0(t_0)\widetilde{B}_0^m(t_0) + \sum_{k=-n/2+1}^{n/2-2} b_k(t_0)\widetilde{S}_k^m(t_0),
\label{eq:Imapprox_mod_fourier}
\end{equation}
where the basis integrals $\widetilde{S}_k^m(t_0)$ are defined and evaluated as in Lemma \ref{lem:mod_fourier_rec} below. As in the open-curve setting, the integral associated with the constant basis function coincides with that of the standard Fourier basis, i.e.~$\widetilde{B}_0^m(t_0)=S_0^m(t_0)$ from \eqref{eq:fourier_integral}. Moreover, the absence of the $a_1$ term in $I_m^{\textrm{mod}}$ is explained by the vanishing of the $\sin(t-a)$ basis integral due to symmetry.

\begin{thm}[Modified Fourier coefficients]\label{thm:mod_coeffs}Let $f:[0,2\pi)\rightarrow\mathbb{R}$ be a $2\pi$-periodic bandlimited function given by a truncated Fourier series
\begin{equation}
f(t)
= \sum\limits_{k=-n/2}^{n/2-1} c_ke^{ikt},
\label{eq:fourier_exp_thm}
\end{equation}
with coefficients $c_k\in\mathbb{C}$, and $n$ even. Then there exists a unique representation of $f$ in the modified basis
\begin{equation}
f(t) =
a_0 + a_1\sin(t-a) + \sum\limits_{k=-n/2+1}^{n/2-2} b_k \sin^2\left(\dfrac{t-a}{2}\right)e^{ikt},
\label{eq:mod_fourier_exp_thm}
\end{equation}
with coefficients $a_0,a_1,\{b_k\}$ given by
\begin{align}
b_{-n/2+1}&= -4e^{-ia}c_{-n/2}, &b_{-n/2+2}&=2e^{-ia}\left(b_{-n/2+1} - 2c_{-n/2+1}\right)
\label{bend}
\\
b_{n/2-2} &= -4e^{ia}c_{n/2-1}, &b_{n/2-3}&= 2e^{ia}\left(b_{n/2-2}-2c_{n/2-2}\right),
\end{align}
\begin{align}
b_k&= e^{ia}\left(2b_{k+1}-e^{ia}b_{k+2}-4c_{k+2}\right),&k&=n/2-4,n/2-3,\dots,1, \\
b_k&= e^{-ia}\left(2b_{k-1}-e^{-ia}b_{k-2}-4c_k\right), &k&=-n/2+3,-n/2+4,\dots,-1,
\label{bk}
\end{align}
\begin{align}
a_0 = d_2-b_0/2,\quad b_0 = -2\left(e^{ia}d_1+e^{-ia}d_3\right),\quad a_1= i\left(e^{ia}d_1-e^{-ia}d_3\right),
\end{align}
with
\begin{align}
d_1 = -\frac{1}{2}b_1+\frac{e^{ia}}{4}b_2+c_1,\quad d_2 = \frac{e^{-ia}}{4}b_{-1}+\frac{e^{ia}}{4}b_1+c_0,\quad d_3 = -\frac{1}{2}b_{-1}+\frac{e^{-ia}}{4}b_{-2}+c_{-1}.
\end{align}
\end{thm}

The structure of \eqref{bend}--\eqref{bk} is that the $n\times n$ matrix mapping $(a_0,a_1,\{b_k\})$
to $\{c_k\}$ is upper-triangular tridiagonal for $k>0$ and
lower-triangular tridiagonal for $k<0$. This means that, although its matrix
inverse is dense, the $(a_0,a_1,\{b_k\})$ coefficients can be solved given
$\{c_k\}$ in linear time by back-substitution (3-term recurrences) starting from the largest $k$ magnitudes and working inwards.
The characteristic values of the recurrences show that they are neutrally stable
(at most quadratic secular growth).

\begin{proof}
Expand the $\sin$ functions in \eqref{eq:mod_fourier_exp_thm} using Euler's formula, and collect coefficients in front of each $e^{ikt}$, $k=-n/2,\ldots,n/2-1$. Then equate the right hand sides of \eqref{eq:fourier_exp_thm} and \eqref{eq:mod_fourier_exp_thm}, and match coefficients.
\end{proof}

\begin{lemma}[Modified Fourier basis integrals]\label{lem:mod_fourier_rec}Let $k\in\mathbb{N}$, $m\in\{2s+1:s\in\mathbb{N}\}$, $t_0=a+ib$, with $a,b\in\mathbb{R}$, $|b|>0$, and $\alpha=e^{-|b|}$, $0<\alpha<1$. Then the basis integrals of the modified Fourier basis in Definition \ref{def:modified_fourier_basis}
\begin{equation}
\widetilde{B}_0^m(t_0) = \int_0^{2\pi} \dfrac{\dt}{|e^{it}-e^{it_0}|^m},\quad \widetilde{B}_1^m(t_0) = \int_0^{2\pi} \dfrac{\sin(t-a)}{|e^{it}-e^{it_0}|^m}\dt,\quad \widetilde{S}_k^m(t_0) = \int_{0}^{2\pi} \dfrac{\sin^2\left(\frac{t-a}{2}\right)e^{ikt}}{|e^{it}-e^{it_0}|^m}\dt
\end{equation}
can be computed as
\begin{equation}
\begin{split}
\widetilde{B}_0^m(t_0) &= \frac{2}{(1-\alpha)^{m-1}}\mu_0^m(\alpha),\qquad \widetilde{B}_1^m(t_0) = 0,\\
\widetilde{S}_k^m(t_0) 
&= C_1\left(-C_2\,\mu_k^m(\alpha) + \frac{2}{m-2}\left(C_3\,\mu_k^{m-2}(\alpha) - C_4\,\mu_{k-1}^{m-2}(\alpha)\right)\right),
\label{eq:Stilderec}
\end{split}
\end{equation}
where
\begin{equation}
C_1=\frac{(1-\alpha)^{3-m}}{2}e^{ika},\qquad
C_2=\frac{(1-\alpha)^2}{2\alpha(1+\alpha^2)},\qquad
C_3=\frac{m/2+k-1}{2\alpha},\qquad
C_4=\frac{m/2+k-2}{1+\alpha^2}.
\end{equation}
Moreover,
\begin{equation}
    \mu_k^m(\alpha) \coloneqq \frac{(1-\alpha)^{m-1}}{2e^{ika}}S_k^m(t_0)
    \label{eq:mukm}
\end{equation}
can be computed through recurrence formulas provided in Appendix \ref{app:rec} for $m=1,3,5$, with the standard basis integrals $S_k^m(t_0)$ found in \eqref{eq:fourier_integral}. Note that $\widetilde{S}_{-k}^m(t_0)=\overline{\widetilde{S}_{k}^m(t_0)}$.
\end{lemma}

\begin{proof}
To derive the recurrence relation for $\widetilde{S}_k^m$ in \eqref{eq:Stilderec}, first expand the $\sin$ functions in the numerator of $\widetilde{S}_k^m(t_0)$ using Euler's formula. This yields a linear combination of the standard Fourier basis integrals with shifted indices, $S_{k-1}^m(t_0)$, $S_{k}^m(t_0)$, $S_{k+1}^m(t_0)$. Invoke the relation \eqref{eq:mukm}, which follows from \cite[Appendix A]{krantz2024}, which leads to a recurrence in terms of $\mu_k^m(\alpha)$. Substitute the recurrence relations for $\mu_k^m(\alpha)$, given in Appendix \ref{app:rec}, into this expression and simplify to obtain the final result.
The formulas for $\widetilde{B}_0^m(t_0)$ follow directly from the corresponding initial values for $\mu_k^m(\alpha)$. Lastly, $\widetilde{B}_1^m(t_0)$ vanishes due to odd symmetry of the integrand about $t=a$.
\end{proof}

\begin{remark}
The real-valued denominator $|e^{it}-e^{it_0}|^m$ handled above, needed for the 3D case, is trickier
than the complex-valued $(e^{it}-e^{it_0})^m$ needed in prior 2D Fourier SSQ work
\cite{afKlinteberg2024,bao2024}. The latter basis integrals need only the residue theorem, whereas the former need complete elliptic integrals (see Appendix~\ref{app:rec} and \cite{krantz2024}).
\end{remark}

This modified basis behaves similarly to its monomial counterpart. For $|k|>1$, the integrals $\widetilde{S}_k^m(t_0)$ are of comparable magnitude to the true value of $I_m$. The accuracy of the approximation is sensitive to the small constant coefficient $a_0(t_0)$, which must be computed to high relative precision. Following the same strategy as in the open-curve case, we set
\begin{equation}
a_0(t_0) = \ftilde(a,\xx)\sigmatilde_a\frac{|e^{ia}-e^{it_0}|^m}{(R(a)^2)^{m/2}},
\label{eq:a0}
\end{equation}
where $\sigmatilde_a$ is obtained from density samples $\{\sigmatilde_j\}$ via trigonometric interpolation, implemented either in its Fourier or barycentric form.

Since the performance of the modified Fourier basis is similar to that of the modified monomial basis described in Section~\ref{ss:open_curves}, we do not include tests on simpler integrals here, and instead turn to an application in the following section.

\begin{remark}[More efficient computation of modified Fourier coefficients]\label{rem:modified_fourier_coeff}While the modified Fourier coefficients
$b_k(t_0)$ in \eqref{eq:Fexp_mod_fourier} can be computed directly using
the FFT and Theorem \ref{thm:mod_coeffs} at a cost of $\mathcal{O}(n \log n)$, the structure of the recurrence formulas of $b_k(t_0)$ renders this somewhat slow in our implementation
(this may not be true for a low-level HPC implementation).
We find that a more efficient approach is to rewrite equation \eqref{eq:Fexp_mod_fourier} in the form
\begin{equation}
G(t) \coloneqq \frac{F(t) - a_0(t_0) - a_1(t_0) \sin(t - a)}{\sin^2\left(\frac{t - a}{2}\right)} \approx \sum_{k = -n/2+1}^{n/2-2} b_k(t_0) e^{ikt}.  
\end{equation}
Now $a_0(t_0)$ is evaluated as in \eqref{eq:a0}, while $a_1(t_0)$ is determined using a standard Fourier interpolation procedure of $F(t)$. The remaining coefficients $b_k(t_0)$ are then obtained via an additional FFT of $n$ samples of the function $G(t)$. Note that this procedure fails if $a$ coincides with a trapezoidal node $t_j$, and may suffer from catastrophic cancellation if $a$ lies very close to one. In such cases, one can revert to computing the coefficients directly via Theorem \ref{thm:mod_coeffs}.
\end{remark}

\subsection{Adjoint method for target-specific quadrature weights}\label{ss:adjoint}
The evaluation of $I_m$ in \eqref{eq:Imswap} according to \eqref{eq:Imapprox_shifted_monomial} and \eqref{eq:Imapprox_mod_fourier} uses the modified basis coefficients. For a 
specific numerator function $F(t)$ in \eqref{eq:F}, i.e.~for a particular density $\sigmatilde$, these are determined using its samples $\FF\coloneqq\{F_j\}_{j=1}^n$. 
In many applications, however, the same target and kernel appear for multiple densities, as in iterative boundary integral solvers. It is then more efficient to precompute \textit{target-specific quadrature weights}
$\LL^m\coloneqq\{L_j^m\}_{j=1}^n$
that depend only on the geometry, kernel, and target location, but can be reused to evaluate $I_m$ for any new smooth density with samples $\boldsymbol{\sigma}\coloneqq\{\sigmatilde(t_j)\}_{j=1}^n$, so that
\begin{equation}
    I_m\approx \boldsymbol{\sigma}^T\LL^m.
    \label{eq:Im_adj_general}
\end{equation}
We refer to this as the \textit{adjoint method}.
Prior adjoint methods have been developed for the SSQ with standard monomial \cite[Section 2.2.2]{AFKLINTEBERG2021} and Fourier bases \cite[Section 2.4]{afKlinteberg2024}; however, in these works the weights,
denoted by $\boldsymbol{\lambda}^m\coloneqq\{\lambda_j^m\}_{j=1}^n$, act on the numerator samples
$\FF$ rather than the density samples $\boldsymbol{\sigma}$, so that $I_m \approx \FF^T \boldsymbol{\lambda}^m$.
We connect the two by noting that $\FF=\mathbf{g}\odot\ssigma$, where the vector $\mathbf{g}$ 
contains geometric factor weights $g_j=g(t_j)$, with $t_j$ denoting the quadrature nodes, and
\begin{equation}
g(t) := \ftilde(t,\xx)\frac{h(t,t_0)^m}{\left(R(t)^2\right)^{m/2}},
\end{equation}
recalling \eqref{eq:F}.
Comparing the two above formulae for $I_m$, one immediately gets
\begin{equation}
\LL^m = \mathbf{g}\odot\llambda^m.
\label{Lfromlambda}
\end{equation}
However, prior methods which compute $\llambda^m$ are unstable in the near-singular settings that we address here, yielding large, oscillatory $\lambda_j^m$ entries. To develop an adjoint method
tailored to the modified bases of Sections \ref{ss:open_curves} and \ref{ss:closed_curves} that retains
their stability, one must bypass \eqref{Lfromlambda} and compute $\LL^m$ directly.
Before describing this, we first recap the prior formulation.

\subsubsection{Standard adjoint formulation}\label{sss:standard_adjoint}
Consider the general setup in \eqref{eq:Imswap}--\eqref{eq:Imapprox}, where the basis function $\varphi$, coefficients $\mathbf{c}$, and basis integrals $\mathbf{B}^m$ may come from either the standard or modified monomial and Fourier bases. 
Let $V$ be the $n\times n$ Vandermonde matrix with entries $V_{jk}=\varphi_k(t_j)$. Then, \eqref{eq:Imapprox} can be written as
\begin{equation}
    I_m \approx \mathbf{c}^T\BB^m = \left(V^{-1}\mathbf{F}\right)^T\mathbf{B}^m = \FF^T\left(\left(V^{-1}\right)^T\BB^m\right).
    \label{eq:Im_adj_std}
\end{equation}
Thus, $\llambda^m=\left(V^{-1}\right)^T\BB^m$, which is the solution to the transposed Vandermonde system
\begin{equation}
    V^T\llambda^m=\BB^m.
    \label{eq:vander_transposed}
\end{equation}
In near-singular cases with nearly vanishing numerators,
this formulation becomes unstable because the entries in $\BB^m$ can attain large magnitudes. For the standard bases, several entries typically have large magnitudes, whereas for the modified bases only the one associated with the constant basis function does. In both cases, these large values give rise to large oscillatory weights $\llambda^m$ and strong cancellation.

\subsubsection{Stabilized adjoint formulation}\label{sss:stablized_adjoint}
The 2nd design principle of a modified basis in Section \ref{s:modified} states that all non-constant basis functions should vanish at the near-vanishing point. Consequently, the vector $\BB^m$ is dominated by its first entry $B_1^m$, associated with the constant basis function. To stabilize the computation
we separate this dominant constant contribution.
Let $\overline{\llambda}^m$ and $\llambda_a$ solve, respectively, each of the linear systems
\begin{equation}
    V^T\overline{\llambda}^m = [0,B_2^m,\dots,B_n^m]^T,\qquad V^T\llambda_a=e_1,
    \label{eq:vander_transposed_stable}
\end{equation}
where $e_1=[1,0,\dots,0]^T\in\mathbb{R}^n$. The vector $\llambda_a$ represents interpolation weights satisfying $\sigmatilde(a)\approx\ssigma^T\llambda_a$, which we in practice obtain directly from the barycentric interpolation formula rather than solving the linear system. By linearity,
\begin{equation}
    \llambda^m = B_1^m\llambda_a + \overline{\llambda}^m.
\end{equation}
A direct substitution of this sum into \eqref{Lfromlambda} and \eqref{eq:Im_adj_general}
would yield $I_m\approx(\mathbf{g}\odot\ssigma)^T(B_1^m\llambda_a+\overline{\llambda}^m)$.
As in Section \ref{sss:stable_eval}, the large magnitude of $B_1^m$ would introduce numerical instability by amplifying any interpolation error in $g(a)\sigmatilde(a)\approx(\mathbf{g}\odot\ssigma)^T\llambda_a$. To avoid this,
for this $k=1$ term only we instead interpolate the density and evaluate $g(a)$ analytically, leading to \eqref{eq:Im_adj_general} with
\begin{equation}
   \LL^m = B_1^m g(a)\boldsymbol{\lambda}_a+\mathbf{g}\odot\overline{\boldsymbol{\lambda}}^m,
   \label{LLm}
\end{equation}
our proposed adjoint weight formula; contrast with \eqref{Lfromlambda}.
This applies to both the modified monomial and modified Fourier bases. For the global modified Fourier basis, where $n$ can be large, the system for $\overline{\llambda}^m$ in \eqref{eq:vander_transposed_stable} can be solved efficiently via FFT acceleration, as described in the following remark.

\begin{remark}[FFT-accelerated adjoint weight construction]\label{rem:fourier_adjoint}For the modified Fourier basis in Definition \ref{def:modified_fourier_basis}, we compute the Fourier coefficients of $\overline{\llambda}^m$ in \eqref{eq:vander_transposed_stable} via a three-term recurrence akin to Theorem \ref{thm:mod_coeffs}, then recover the physical-space weight by an inverse FFT, hence reducing the weight construction cost from $\mathcal{O}(n^2)$ to $\mathcal{O}(n\log n)$.

Define the discrete Fourier transform (DFT) of $\overline{\llambda}^m$ as (assuming $n$ even)
\begin{equation}
    \widehat{\Lambda}_\ell = \sum_{j=1}^n \overline{\lambda}_j^me^{i\ell t_j},\quad\ell\in\{-n/2,\dots,-1,0,1,\dots,n/2-1\}.
\end{equation}
The first row of the transposed Vandermonde system in \eqref{eq:vander_transposed_stable} gives $\sum_{j=1}^n\overline{\lambda}_j^m=0$, or equivalently, $\widehat{\Lambda}_0=0$.

Expanding the sine-functions of the modified basis and identifying the DFT of $\overline{\llambda}^m$ leads to $e^{ia}\widehat{\Lambda}_{-1}-e^{-ia}\widehat{\Lambda}_{1}=0$ for the second row, and for the remaining rows,
\begin{equation}
    \frac{1}{2}\widehat{\Lambda}_k-\frac{e^{ia}}{4}\widehat{\Lambda}_{k-1}-\frac{e^{-ia}}{4}\widehat{\Lambda}_{k+1} = \widetilde{S}_k^m, \quad k\in\{-n/2+1,\dots,-1,0,1,\dots,n/2-2\}.
    \label{eq:lambdahat}
\end{equation}
The coefficients $\widehat{\Lambda}_{-1}$, $\widehat{\Lambda}_{0}$, and $\widehat{\Lambda}_{1}$ obtained from the first two relations initialize the recurrence above from which we get the remaining coefficients. The weights $\overline{\llambda}^m$ are then recovered by an inverse DFT.
\end{remark}

\subsection{Computational cost}\label{ss:computational_cost}
The primary computational costs of SSQ are the computation of the basis coefficients and the evaluation of the corresponding basis integrals. For the standard monomial and Fourier bases, the basis integrals can be evaluated in $\mathcal{O}(n)$ operations using recurrence relations, assuming each term is computed in constant time. The basis coefficients are obtained at a cost of $\mathcal{O}(n^2)$ for monomials and $\mathcal{O}(n\log n)$ for Fourier bases via the FFT. Note that $n$ is typically small and fixed for the monomials, as refinement is achieved by increasing the number of panels. 

The modified bases introduced in this work retain the same asymptotic complexities, with only modest increases in the constant factors. For monomials, the only additional step is a barycentric interpolation used to stabilize the constant coefficient. For Fourier bases, the extra work consists of one additional recurrence, an interpolation in Fourier or barycentric form for the stabilization step, and a Fourier interpolation to compute the coefficients as described in Remark~\ref{rem:modified_fourier_coeff}. 

In the adjoint formulation, the basis coefficients are replaced by adjoint weights. For the standard monomial and Fourier bases, this does not change the computational cost, and the same holds for the modified monomial basis using the stabilized adjoint formulation. The only additional cost occurs for the modified Fourier basis, where the adjoint weights are computed as described in Remark~\ref{rem:fourier_adjoint}, introducing one additional recurrence relative to the non-adjoint case.

In summary, the modified approach adds only a minor overhead beyond standard SSQ, both in the direct and adjoint formulation.

\subsection{When to use the modified bases}\label{ss:when_to_use_mod_basis}
For monomial bases on the interval $[-1,1]$, the translated (modified) monomial basis can safely be used whenever the complex root $t_0=a+ib$ satisfies $|a|\leq 1$, though it may still offer benefits for slightly larger $|a|$. The magnitude of the modified basis integrals grows roughly as $(1+|a|)^k$ for $k=1,\dots,n$, which for large $n$ and $|a|>1$ results in large, oscillatory weights and increased cancellation. This is not problematic in practice, since the standard monomial basis performs well in this regime where cancellation is mild. For periodic problems, the modified Fourier basis remains stable for all target locations. 

For efficiency, the use of the modified bases can be limited to $1/|\rr|^3$ and $1/|\rr|^5$ ($m=3,5$)-type kernels with nearly-vanishing numerators and very close evaluations, e.g.~when $|b|$ falls below a threshold such as $10^{-2}$, since cancellation only arises for near-singular targets and the modified formulation incurs a slightly higher computational cost.

\section{Numerical experiments}\label{s:numerical_experiments}
In Section \ref{sss:test_mod_monomial}, we demonstrated on a prototype integral that the standard monomial basis yields growing error as the distance between the target point and the curve $\Gamma$ decreases, whereas the modified (translated) basis maintains near machine-precision accuracy across all distances. We now compare the performance of these basis choices in a more realistic setting, namely for the close evaluation of Stokes flow around a slender filament, using the classical slender body approximation. We test panel-wise (open arc) and Fourier (global) quadratures.

Our panel-wise test case (Section~\ref{ss:tangle}) was previously considered for singularity swap quadrature (SSQ) in \cite[Section 5.2]{AFKLINTEBERG2021}, to which we refer for full implementation details.
Our Fourier test (Section~\ref{ss:starfish}) is new.
The code used to generate all numerical results in this paper is available at \cite{code}, which builds on the codes provided in \cite{linequad}. All code for this paper is written and run in MATLAB, and is not optimized for runtime. The computer used for the numerical results of this section has a 3.4 GHz quad-core Intel i7-6700 CPU.

\subsection{Setup}\label{ss:Stokes}
The classical slender body approximation models the Stokes velocity field $\uu(\xx)\in\mathbb{R}^3$ at a point $\xx\notin\Gamma$, where $\Gamma$ is the centerline of a slender body with radius $\varrho\ll1$, as
\begin{equation}
\uu(\xx) = \int_\Gamma \left(\mathcal{S}(\xx-\yy)+\frac{\varrho^2}{2}\mathcal{D}(\xx-\yy)\right)\boldsymbol{\sigma}(\yy)\ds(\yy),
\label{eq:slender_layer_potential}
\end{equation}
where $\boldsymbol{\sigma}(\yy)$ is the force density along the filament, $\rr\coloneqq\xx-\yy$, the kernels
\begin{equation}
\mathcal{S}(\rr) = \frac{\mathbf{I}}{|\rr|} + \frac{\rr\rr^T}{|\rr|^3}, \qquad \mathcal{D}(\rr) = \frac{\mathbf{I}}{|\rr|^3} - 3\frac{\rr\rr^T}{|\rr|^5},
\label{eq:Stokes}
\end{equation}
are the so-called Stokeslet and doublet, respectively.

To apply the SSQ method, we split the integrand according to its singular behavior
\begin{equation}
\uu(\xx) = I_1+I_3+I_5,
\end{equation}
where
\begin{align}
    I_1 &= \int_\Gamma \frac{\boldsymbol{\sigma}(\yy)}{|\rr|}\ds(\yy),\label{eq:I1} \\
    I_3 &= \int_\Gamma \frac{\left(\rr\rr^T+\varrho^2\mathbf{I}/2\right)\boldsymbol{\sigma}(\yy)}{|\rr|^3}\ds(\yy),\label{eq:I3} \\
    I_5 &= -\frac{3\varrho^2}{2}\int_\Gamma \frac{\rr\rr^T\boldsymbol{\sigma}(\yy)}{|\rr|^5}\ds(\yy).\label{eq:I5}
\end{align}
For a target point located a distance $d$ from $\Gamma$, the standard SSQ method computes the quadrature weights for each term $I_1$, $I_3$, and $I_5$ using the standard monomial or Fourier basis through the adjoint formulation described in Section \ref{sss:standard_adjoint}. We refer to this baseline approach simply as SSQ.

The strategy referred to as translated SSQ (TSSQ) replaces the standard basis with the modified (translated) one according to the criteria outlined in Section~\ref{ss:when_to_use_mod_basis} and employs the stabilized adjoint formulation described in Section \ref{sss:stablized_adjoint}. In practice, this is done only for the $I_3$ and $I_5$ terms, since numerical tests indicate that the $1/|\rr|$ ($m=1$) singularity does not incur significant cancellation error.

In all numerical experiments, following \cite{AFKLINTEBERG2021}, we set $\boldsymbol{\sigma}(\yy)=\yy$ and $\varrho=10^{-3}$. The error in evaluating $\uu(\xx)$ is computed as the maximum componentwise absolute error between the numerical and reference solution, normalized by the maximum norm of the reference solution. This quantity is referred to as the \textit{relative error}. 
The reference solution is computed using the same adaptive quadrature approach as in \cite{AFKLINTEBERG2021}. 
For a given target point $\xx$, panels are recursively subdivided until each panel $\Gamma_i$ satisfies $\min_{\yy\in\Gamma_i}|\xx-\yy|~\geq~H\eta_i$, where $H$ is a refinement parameter and $\eta_i$ denotes the arc length of $\Gamma_i$.
Note that with $\varrho=10^{-3}$, targets with $d<10^{-3}$ lie inside the fiber and are not physically relevant. We nevertheless include such cases in our tests, since they provide a challenging stress test of the quadrature and remain relevant for the Stokeslet contribution.

\subsection{Long thin filament}\label{ss:tangle}
As a first example we consider the tangled long thin filament $\Gamma$ parametrized by $\boldsymbol{\gamma}:[0,1]\rightarrow\mathbb{R}^3$, also used\footnote{See \cite[Section 5.2]{AFKLINTEBERG2021} for the expression of $\boldsymbol{\gamma}$ and how the filament is generated.}
in \cite[Section 5.2]{AFKLINTEBERG2021}, and illustrated in Figure \ref{fig:long_filament}.
The curve is discretized adaptively using composite Gauss--Legendre quadrature with $n=16$ points per panel. The discretization is refined until the Legendre expansion coefficients $\hat{\mathbf{s}}=\{\hat{s}_k\}_{k=1}^n$ of the speed function $s(t)=|\boldsymbol{\gamma}'(t)|$ satisfy 
\begin{equation}
    \max(|\hat{s}_{n-1}|,|\hat{s}_n|)<\epsilon\|\hat{\mathbf{s}}\|_\infty,
    \label{eq:resolution_critera}
\end{equation}
for a given tolerance $\epsilon$. 

Errors are measured relative to a reference solution computed using the adaptive quadrature strategy, but with 18-point Gauss--Legendre panels, $H=1$, and $\epsilon=5\times 10^{-14}$. This choice ensures that the reference discretization is both different from and more finely resolved than the grids used for SSQ and TSSQ.

Both SSQ and TSSQ use upsampling to $32$ Gauss--Legendre points for all targets whose complex root $t_0=a+ib$ lies within the Bernstein ellipse of radius 3 (ensuring full accuracy). Outside this region, standard Gauss--Legendre quadrature is used. In TSSQ, the translated monomial basis is activated when $|b|\leq 10^{-2}$.

For each distance $d$,
we evaluate the integral \eqref{eq:slender_layer_potential} at 1000 targets randomly placed to be exactly $d$ from the filament $\Gamma$, and report the minimum, maximum, and mean relative error achieved by standard SSQ and TSSQ. 
Figure \ref{fig:long_filament_err_vs_dist_tol4} presents the results for $\epsilon=10^{-4}$, where the filament is discretized using 59 panels (944 nodes), and Figure \ref{fig:long_filament_err_vs_dist_tol6} for $\epsilon=10^{-6}$, where 91 panels (1456 nodes) are used.

Standard SSQ exhibits a growing relative error of order $\mathcal{O}(1/d^2)$ regardless of the tolerance, indicating a nonconvergent unbounded trend as $d\rightarrow0$, while TSSQ achieves errors near or below $\epsilon$, largely independent of $d$. In particular, for $\epsilon=10^{-6}$, the maximum TSSQ error is $10^{-7}$ over all distances $d>10^{-7}$, while standard SSQ is up to 9 digits worse than TSSQ over this distance range.

To assess performance, we also measured the rate of computing target-specific quadrature weights via the adjoint method of Section~\ref{sss:stablized_adjoint}. Accumulating timings over both tests, the methods achieved approximately $4.3\times 10^4$ and $4.3\times 10^4$ weight sets per second for SSQ and TSSQ, respectively. The additional cost for TSSQ is therefore negligible in the panel case.

\begin{figure}[t!]
\centering
\begin{subfigure}[t]{0.5\textwidth}
\includegraphics[width=\linewidth]{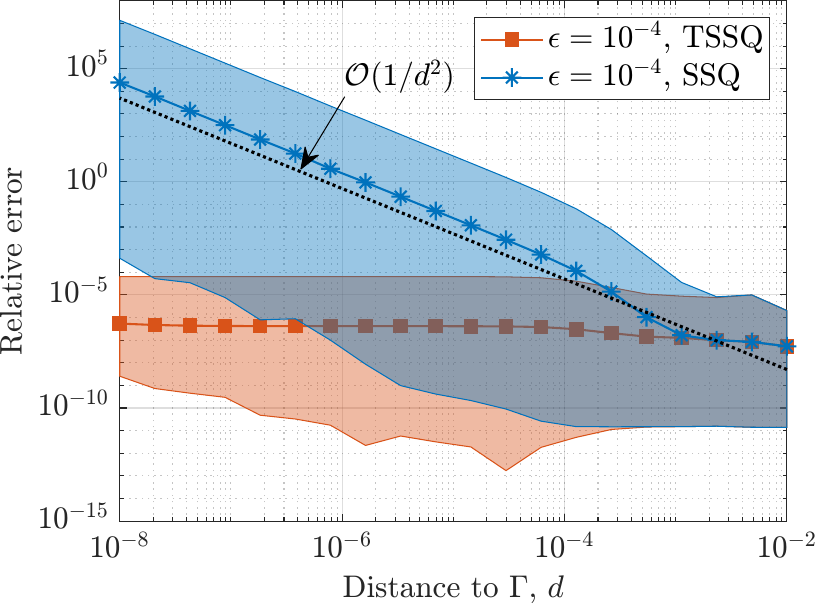}
\caption{}
\label{fig:long_filament_err_vs_dist_tol4}
\end{subfigure}%
\begin{subfigure}[t]{0.5\textwidth}
\includegraphics[width=\linewidth]{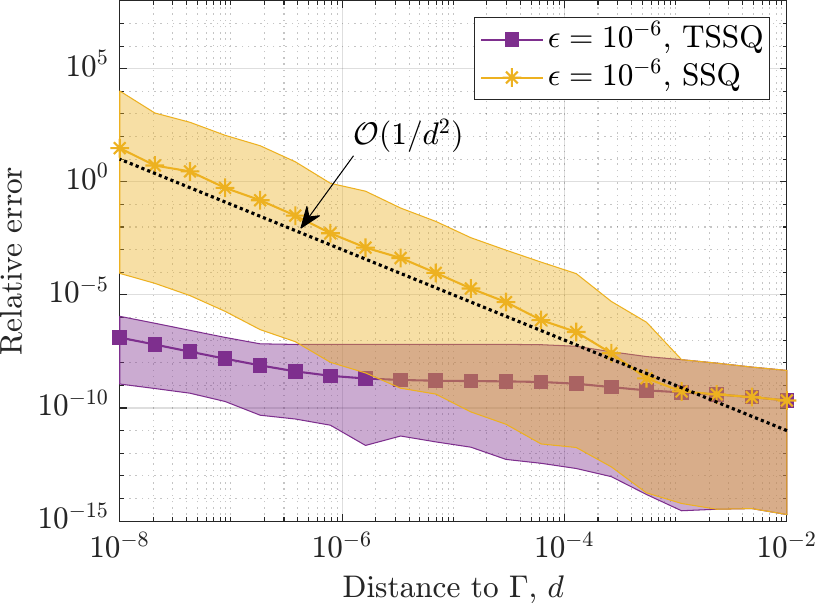}
\caption{}
\label{fig:long_filament_err_vs_dist_tol6}
\end{subfigure}%
\hfill
\begin{subfigure}[t]{0.42\textwidth}
\includegraphics[trim={1.7cm 0.3cm 2.7cm 0.8cm},clip,width=\linewidth]{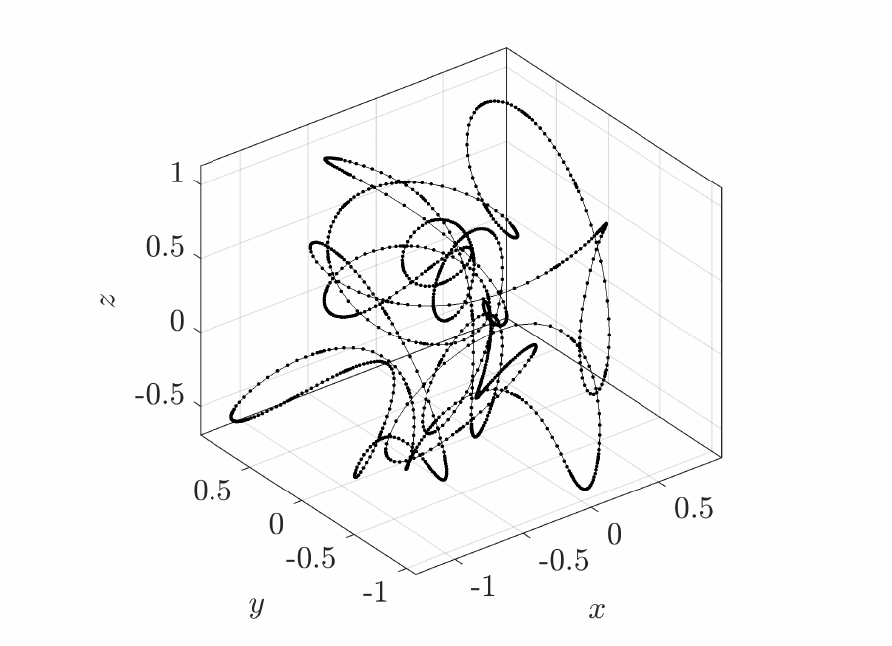}
\caption{}
\label{fig:long_filament}
\end{subfigure}
\caption{Panels (a) and (b) compare the minimum, maximum, and mean relative error of the standard SSQ and translated SSQ (TSSQ) methods in evaluating the flow field \eqref{eq:slender_layer_potential} for different tolerances $\epsilon$ at 1000 random target points sampled at each distance $d$ from the filament $\Gamma$ shown in panel (c), with dots illustrating the discretization nodes (when $\epsilon=10^{-6}$).}
\label{fig:tangle}
\end{figure}

\subsection{Deformed thin starfish}\label{ss:starfish}
We now investigate the performance of the standard (Fourier) basis SSQ and corresponding modified SSQ method (which we also refer to as TSSQ) for close evaluation of the line potential in \eqref{eq:slender_layer_potential}, where the curve $\boldsymbol{\gamma}:[0,2\pi)\rightarrow\mathbb{R}^3$ is the closed curve given by
\begin{equation}
    \boldsymbol{\gamma}(t) = \Big((1+0.3\cos(5t))\cos(t),~(1+0.3\cos(5t))\sin(t),~2\sin(t)\Big),\quad t\in[0,2\pi).
\end{equation}
This geometry, which we refer to as the \textit{deformed thin starfish}, is shown in Figure \ref{fig:deformed_starfish_curve}. 

The curve is discretized using $n=512$ nodes periodically equispaced in $t$, chosen so that the resolution criterion in \eqref{eq:resolution_critera} holds with $\epsilon=5\times 10^{-14}$, where the coefficients $\{\hat{s}_k\}$ now denote the Fourier coefficients of the speed function $s(t)=|\boldsymbol{\gamma}'(t)|$. For each fixed distance $d$ from the curve, we randomly generate 5000 target points and evaluate the potential at these locations, reporting the minimum, maximum, and mean relative error over the 5000 targets. 
Errors are measured relative to a reference solution computed using the same adaptive quadrature strategy as before, but with 34-point Gauss--Legendre panels, $H=4$, and $\epsilon=5\times 10^{-14}$.

Figure \ref{fig:deformed_starfish_err_vs_dist} shows the results. Standard SSQ exhibits an error that grows like $\mathcal{O}(1/d^2)$, consistent with the panel-based discretization in Figure \ref{fig:tangle}. To reach the target tolerance $\epsilon=10^{-10}$, it is necessary to correct both the $1/|\rr|^3$ and $1/|\rr|^5$ terms of the line potential in \eqref{eq:I3} and \eqref{eq:I5}, respectively. With these corrections, TSSQ achieves the desired mean accuracy of 10 digits for all distances at which the reference solution is sufficiently precise to measure it ($d>3\times 10^{-6}$). The maximum errors are at most one digit worse. In contrast, SSQ is up to 7 digits less accurate over the same distance range. 
Note that changing the number of discretization points $n$ would not fundamentally change this result, provided that $n$ is sufficiently large to resolve the geometry and density. It would, however, shift the distance range in which special quadrature (SSQ or TSSQ) must be activated to achieve the target tolerance.

In terms of efficiency, the methods compute approximately $9.5\times 10^3$ and $5.0\times 10^3$ sets of target-specific quadrature weights per second for SSQ and TSSQ, respectively. The halved throughput for TSSQ mainly stems from the extra recurrence in \eqref{eq:lambdahat}, together with the large $n=512$ and
the fact that all $10^5$ target points require the modified basis.
It is likely that an HPC implementation in a lower-level language could reduce this throughput gap.

\begin{figure}[t!]
\centering
\begin{subfigure}[t]{0.5225\textwidth}
\includegraphics[width=\linewidth]{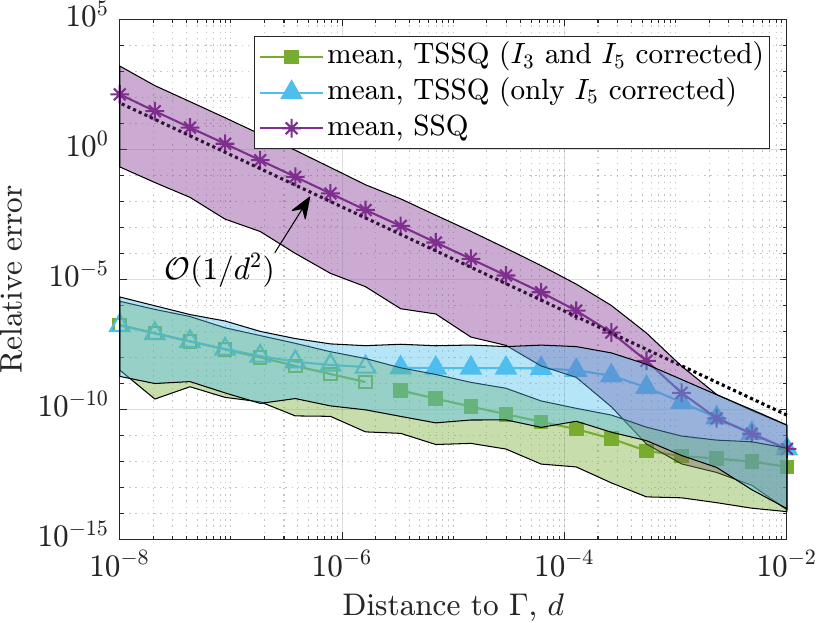}
\caption{}
\label{fig:deformed_starfish_err_vs_dist}
\end{subfigure}%
\begin{subfigure}[t]{0.32\textwidth}
\includegraphics[trim={2.5cm 0.25cm 4.3cm 0.8cm},clip,width=\linewidth]{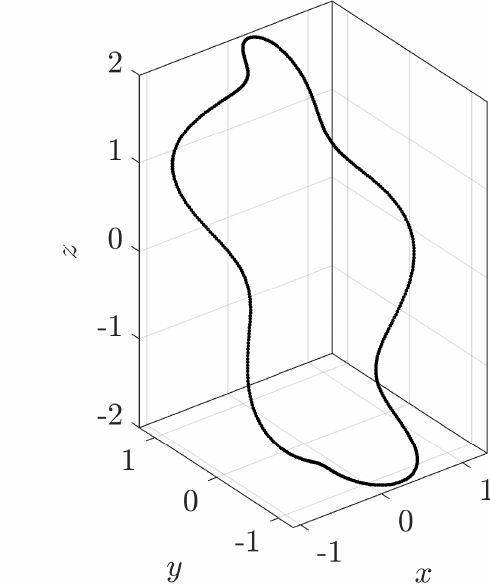}
\caption{}
\label{fig:deformed_starfish_curve}
\end{subfigure}
\caption{Panel (a) compares the minimum, maximum, and mean relative error of the standard SSQ and translated (TSSQ) methods in evaluating \eqref{eq:slender_layer_potential} at 5000 random points sampled at each distance $d$ from the globally discretized closed curve $\Gamma$ of panel (b). 
Filled markers indicate that the reference solution is accurate to within $10^{-10}$, and unfilled markers denote distances ($d<3\times 10^{-6}$) where it is not.}
\label{fig:deformed_starfish}
\end{figure}

\section{Conclusions}\label{s:conclusion}
We have introduced an enhanced version of \textit{singularity swap quadrature} (SSQ) \cite{AFKLINTEBERG2021,afKlinteberg2024,bao2024,krantz2024} for accurate evaluation of potentials near smooth curves in three dimensions. The method retains the structure of SSQ, but with the crucial difference that the choice of interpolation basis depends on the evaluation target point.

The standard bases---monomials for open curves and Fourier modes for closed ones---perform well when the kernel exhibits a mild singularity and does not contain near-vanishing factors. However, when the numerator nearly vanishes at the evaluation point, as it does if it originates from a derivative or tensorial Green's kernel, 
the standard bases lead to quadrature sums with oscillatory basis contributions
that are much larger than the true value of the integral.
This results in \textit{catastrophic cancellation} and severe loss of precision, with errors in our tests growing as the inverse square of the minimum distance between the target point and the curve.

Our remedy is simple and effective, with the central idea being that when the kernel vanishes, all but one of the interpolation bases should too. For open curves, we translate the monomials to vanish at the point of near-vanishing in parameter space. For closed curves, we replace the standard Fourier basis with a modified one tailored to vanish at the same point.
A key ingredient is the stable evaluation of the coefficient of the now-dominant constant term.
We refer to the approach as \textit{translated SSQ} (TSSQ).

The improvements are prominent, as we show through several numerical experiments, including applications from slender body theory where accurate close evaluations are often required.
For target distances as small as $10^{-8}$, TSSQ achieves errors up to ten orders of magnitude smaller than standard SSQ at the same computational cost.
Yet in the more physically relevant range $10^{-5}<d<10^{-3}$, TSSQ still gives a 2--5 digit improvement over SSQ.
We also derived an adjoint formula for stable target-specific quadrature weights that allows TSSQ to act on samples of arbitrary new smooth densities.

TSSQ is likely to find utility in line-integral potential solvers for PDEs with wire or fiber geometries (especially when robustness or high accuracy is needed), 
but also as a subroutine within quadrature schemes for solving boundary integral equations on general surfaces in three dimensions.
TSSQ runs as fast as SSQ while remaining stable and accurate where SSQ fails, and it can be applied wherever SSQ is used, making it a safe and reliable choice.

\section*{Acknowledgments}
Krantz and Tornberg acknowledge support from the Swedish Research Council under grant 2023-04269. The authors thank Ludvig af Klinteberg (Mälardalen University) for making the code from the original SSQ paper publicly available, which included an implementation of the recurrence formulas in Lemma \ref{lem:mod_monomial_rec}. Krantz also gratefully acknowledges the hospitality of the Center for Computational Mathematics at the Flatiron Institute, particularly Manas Rachh, and thanks the Simons Foundation for its support during this visit.
It was during this visit that the authors had the initial ideas for this work.
The Flatiron Institute is a division of the Simons Foundation.
Part of this work was carried out during a program at Institut Mittag-Leffler in Djursholm, Sweden, during the fall semester of 2025, with support from the Swedish Research Council under grant 2021-06594.

\appendix

\section{Recurrence formulas for \boldmath{$\mu_{k}^{m}(\alpha)$}}\label{app:rec}Here we provide the necessary recurrence formulas to compute the modified Fourier basis integrals needed in Lemma \ref{lem:mod_fourier_rec}. The following formulas were derived in \cite[Lemma 2.3]{krantz2024}.

Define the complete elliptic integrals of the first and second kind as
\begin{equation}
    K(\alpha^2) = \int_0^{\pi/2} \frac{\textrm{d}\theta}{\sqrt{1-\alpha^2\sin^2(\theta)}},\qquad E(\alpha^2) = \int_0^{\pi/2}\sqrt{1-\alpha^2\sin^2(\theta)}~\textrm{d}\theta,
    \label{eq:ellipticKE}
\end{equation}
respectively. Then $\mu_k^m(\alpha)$ defined in \eqref{eq:mukm} can be expressed as
\begin{equation}
    \mu_k^m(\alpha) =
    \begin{cases}
        \dfrac{1+\alpha^2}{\alpha}\dfrac{2(k-1)}{2k-1}\mu_{k-1}^m(\alpha) - \dfrac{2k-3}{2k-1}\mu_{k-2}^m(\alpha), & m=1\text{ and } k=2,3,\dots,\\
        \dfrac{1+\alpha^2}{2\alpha}\mu_{k-1}^{m}(\alpha) - \dfrac{(1-\alpha)^2}{2\alpha}\dfrac{m/2+k-2}{m/2-1}\mu_{k-1}^{m-1}(\alpha), & m>1\text{ and } k=1,2,\dots.
    \end{cases}
    \label{eq:mu_rec}
\end{equation}
The initial values for $\mu_k^m(\alpha)$ for $m=1,3,5$ are
\begin{align}
\mu_0^{1}(\alpha) &= 2K(\alpha^2), \quad \mu_1^{1}(\alpha) = \frac{2}{\alpha}\left(K(\alpha^2)-E(\alpha^2)\right),\label{eq:mu0p1}\\
\mu_0^{3}(\alpha) &= \frac{2}{1+\alpha}\left(\frac{2}{1+\alpha}E(\alpha^2) - (1-\alpha)K(\alpha^2)\right),\label{eq:mu0p3}\\
\mu_0^{5}(\alpha) &= \frac{2}{3(1+\alpha)^4}\left(8(1+\alpha^2)E(\alpha^2)-(1-\alpha)(1+\alpha)(5+3\alpha^2)K(\alpha^2)\right).\label{eq:mu0p5}
\end{align}

\clearpage
\bibliographystyle{siamplain}
\bibliography{references}

\end{document}